\documentclass[12pt,a4paper]{article}
\usepackage{fullpage}
\usepackage[utf8x]{inputenc}
\usepackage{hyperref}
\usepackage[square]{natbib}
\usepackage{amsmath,amsthm,amssymb}

\newtheorem{theorem}{Theorem}
\newtheorem{corollary}{Corollary}
\newtheorem{proposition}{Proposition}

\newtheorem{preliminary}{Preliminary}

\theoremstyle{definition}
\newtheorem{definition}{Definition}

\newcommand{\N}{\mathbb{N}}
\newcommand{\NZ}{\N_0}
\glossary{$\NZ$: natural numbers including 0}

\newcommand{\C}{\mathbb{C}}

\newcommand{\eps}{\varepsilon}

\DeclareMathSymbol{\restr}{\mathpunct}{AMSa}{"16}

\newcommand{\mt}[1]{\quad\mbox{#1}\quad}
\newcommand{\abs}[1]{\left|#1\right|}

\author{Henryk Trappmann}
\title{The Intuitive Logarithm}
\begin{document}
\bibliographystyle{plainnat}
\citestyle{plain}
\maketitle
\begin{abstract}
We introduce the {\em intuitive} method to select an analytic Abel
function of an analytic function f at a non-fixpoint. Due to the
complexity of this method by involving matrix inversion of 
increasing size there is little known about its convergence. 

We show its convergence in the simplest but still complicated case
f(x)=bx.
We show that the obtained Abel function is, as expected, the
logarithm to base b, independent on its development point. As
a by-product we obtain a new polynomial approximation sequence for the
logarithm to base b.
\end{abstract}

\section{Introduction}
In the context of discussions of non-integer iterates of the
exponential function there emerged a method which --- in case of
success --- selects an analytic Abel function at a non-fixpoint.

For a function $f\colon G\to \C$ we call a function $\alpha\colon
G\cap f^{-1}(G)\to \C$ {\em Abel function} of $f$ iff it satisfies the Abel equation
\begin{align}\label{eq:Abel}
  \alpha(f(z))=\alpha(z)+1
\end{align}
on its domain.
Abel functions are an essential tool for non-integer/continuous
iteration. For $\alpha$ being bijective in an appropriate way one can
define iterates by 
\begin{align*}
  f^{[t]}(z)=\alpha^{-1}(t+\alpha(z)),
\end{align*}
they satisfy
\begin{align*}
  f^{[1]}&=f & f^{[s+t]}=f^{[s]}\circ f^{[t]}.
\end{align*}
for $t$ being contained in some additive semigroup of $\C$ containing $1$.
Particularly $f^{[n]}$ is the $n$-times iteration/composition of the
function $f$ for positive integers $n$.

We always consider Abel functions up to an additive constant, as one
can see that if $\alpha=\alpha_1$ satisfies \eqref{eq:Abel} then also
$\alpha(z)=\alpha_1(z)+c$ satisfies \eqref{eq:Abel}. 
However even up to an additive constant analytic Abel
functions are not uniquely determined: If $\theta$ is an analytic
1-periodic function then 
$\alpha(z)=\theta(z)+\theta(\alpha_1(z))$ also satisfies the Abel
equation \eqref{eq:Abel} which is easy to verify.

There is an exhaustive theory about existence and uniqueness of
analytic iterations (and the corresponding analytic Abel functions)
developed at a fixpoint of $f$, see e.g.\ Szekeres \citep{szekeres:regular},
Écalle \cite{Ecalle:InvariantsHolomorphes} or the monograph
\cite{Kuczma:IterativeFunctionalEquations}. We refer to this method as
{\em regular iteration} following Szekeres.

As the exponential function $e^x$ has no real fixpoint, regular iteration is not applicable
and quite different methods emerged aimed at obtaining real-analytic Abel functions
anyway
\cite{Kneser:Reelle,ref53.0304.02,Walker:infinitely,Aldrovandi:Carleman}. ``Methods''
here includes recipes with unverified outcome. 

For example some years ago Peter Walker   
\cite{walker:solutions} was proposing a way to calculate the 
powerseries of an Abel function $\alpha$ of the exponential
$f(x)=e^x$ by solving an infinite linear equation system. His method
was independently rediscovered in the lay-mathematical community
(Andrew Robbins \cite{AndrewRobbins:2005:}), which documents a great interest 
for these kind of questions.

His method works as follows: We consider the Abel equation
\begin{align*}
  \alpha\circ f= 1 + \alpha
\end{align*}
with formal powerseries $\alpha$ and $f$ (in the hope that we obtain
$\alpha$ with $f(0)$ inside its convergence disk). We write the coefficient of
$x^n$ in the formal powerseries $f$ as $f_n$. Then the formula for
powerseries composition is $(\alpha\circ f)_m=\sum_{n=0}^\infty \alpha_n {f^n}_m$ where
${f^n}_m$ is the coefficient of $x^m$ in the $n$-th power of $f$.
The Abel equation can then be written as the infinite equation system
in the coefficients of $\alpha$:
\begin{align*}
  \sum_{n=0}^\infty \alpha_n {f^n}_m &=  I_{m,0} + \alpha_m & I_{m,n}
  &:= \begin{cases}1 & m=n\\0&m\neq n\end{cases}
\end{align*}
If we subtract $\alpha_m$ on each line $m$ then we get the standard
form of an infinite linear equation system:
\begin{align}\label{eq:Abel_equation_system}
  \sum_{n=1}^\infty \underbrace{\left({f^n}_m -
       I_{m,n}\right)}_{A_{m,n}} \alpha_n &=  I_{m,0} & m &\ge 0
\end{align}
The first column $A_{m,0}$ is always 0 as ${f^0}_0=1$ and ${f^0}_m=0$ for $m\ge
1$. That's why we start the sum (and $\alpha_n$) with $n=1$ (and we know anyway that
$\alpha$ may be determined merely up to $\alpha_0$).

Still this equation system must have infinitely many solutions, if
any, of $\alpha$ containing $f(0)$ in its convergence disk as we
explained before. 
The {\em intuitive} method to solve this equation
system --- and hence to select one of the infinitely many solutions
--- is to solve the to $N\times N$ truncated equation systems 
\begin{align}\label{eq:Abel_equation_system_truncated}
  \sum_{n=1}^N \underbrace{\left({f^n}_m -
       I_{m,n}\right)}_{A_{m,n}} \alpha^{(N)}_n &=  I_{m,0} &
   m=0,\dots,N-1
\end{align}
for increasing $N$ with the solution
$\alpha^{(N)}_1,\dots,\alpha^{(N)}_N$ in the hope that $\alpha_n:=\lim_{N\to\infty}
\alpha^{(N)}_n$ exists for every $n$ and the so obtained
coefficient sequence (or formal powerseries) $\alpha$ is a solution of the
untruncated equation system (or Abel equation) and has non-zero convergence radius.
\begin{definition}[$\restr_N$, intuitive]\label{def:intuitive}
  For an infinite linear equation system $Ax=b$ we denote the to $N$
  rows and $N$ columns truncated matrix with $A\restr_N$, similarly
  we denote truncated vectors.

  We call $(x_m)_{m\in\N}$ the {\em intuitive solution} if the limit
  $x_m:=\lim_{N\ge m,N\to\infty}\left(A\restr_N^{-1}
    b\restr_N\right)_m$ exists for each $m$.

  We call \eqref{eq:Abel_equation_system} the (infinite linear)
  {\em Abel equation system} of $f$ (developed at 0). We call its intuitive solution with
  $\alpha_0=0$ the {\em intuitive (formal) Abel powerseries} if
  existing. If $\alpha$ is analytic at 0 
  (i.e.\ having non-zero convergence radius) we call the corresponding
  analytic function just the {\em intuitive Abel function}, written as $\mathcal{A^I}[f]$.

  We call the intuitive Abel function $\beta$ of $g(x):=f(x+s)-s$
  (for $s$ inside the convergence disk of $f$) the intuitive Abel
  function of $f$ {\em developed at $s$}, written as
  $\mathcal{A}^{\mathcal{I}}_s[f]$. We call the function 
  $\beta(x-s)$ the {\em $s$-intuitive Abel function} of $f$.
\end{definition}

The last part perhaps needs some explanation. Generally if $\beta$ is
an Abel function of a conjugation $g=h^{-1}\circ f\circ h$ then
$\alpha=\beta\circ h^{-1}$ is an Abel function of $f$:
\begin{align}\label{eq:conjugation}
 \alpha\circ f = \beta\circ h^{-1} \circ f=\beta \circ g \circ h^{-1} = (1 + \beta) \circ h^{-1} = 1+\alpha. 
\end{align}
Particularly this is true for
shift conjugations $g(x)=f(x+s)-s$. 
In this case the above $\beta$ is the intuitive Abel function
developed at $s$. By definition $\beta(s)=0$. The above $\alpha$ is the
$s$-intuitive Abel function of $f$. 

Several until now unanswered questions arise here: For which $f$ does
the intuitive Abel powerseries exist (i.e.\ the coefficients
converge)? Is $\mathcal{A}^{\mathcal{I}}_s$ independent 
on $s$ in the sense that
$\mathcal{A}^{\mathcal{I}}_s[f](x)-\mathcal{A}^{\mathcal{I}}[f](x+s)$
is constant in $x$? Is it generally
invariant under conjugation, i.e.\ for which 
$h$ is $\mathcal{A^I}[h^{-1}\circ f\circ
h]-\mathcal{A^I}[f]\circ h$ constant? How does it
relate to regular iteration at a (nearby) fixpoint?

Besides the above questions there also arises the question 
whether this procedure gives the expected results for known elementary
Abel functions of $f$. The most basic example being $f(x)=b x$ with the
Abel function $\alpha(x)=\log_b(x)$.

\section{Intuitive Abel function of f(x)=bx}
$f$ has already the fixpoint 0 and it should be noted that 
the {\em regular} Abel function developed at this fixpoint is
indeed $\log_b$. The {\em intuitive} Abel function can however not be
directly developed at fixpoint 0, because in this case the first line of
our equation system is: $0\alpha_1+0\alpha_2+\dots = 1$.

So we proceed by calculating the intuitive Abel function $\beta$
developed at $s\neq 0$, i.e.\ the intuitive Abel function of the shift
conjugation $g(x)=b\cdot(x+s)-s$ which gives the $s$-intuitive Abel  
function $L_{b,s}(x)=\beta(x-s)$. 
We will later see that $L_{b,s}$ is independent on $s$ up to an
additive constant.

\subsection{Solving the truncated linear equation system with a
  recurrence}
In this subsection we solve the truncated equation system
\eqref{eq:Abel_equation_system_truncated} for $f$ being the above
given shift-conjugation $g$ obtaining the recursive formula
\eqref{eq:recurrence}. We call the solutions $\beta^{(N)}$ instead of
$\alpha^{(N)}$. 

In order to determine the occuring ${g^n}_m$ (which, recall, is the $m$-th
coefficient of the $n$-th power of $g$) we calculate:
\begin{align*}
  g(x)^n=(bx+\underbrace{s(b-1)}_{=:d})^n=\sum_{k=0}^n \binom{n}{k} d^{n-k}b^kx^k
\end{align*}
hence the Matrix $A$ in \eqref{eq:Abel_equation_system} is given by
subtracting the identity matrix from the matrix given by
\begin{align*}
  B_{m,n}&=\binom{n}{m} d^{n-m}b^m&\mt{e.g.\ } B\restr_4&=\begin{pmatrix}
    1 & d & d^2 & d^3\\
    0 & b & 2 d b & 3 d^2 b\\
    0 & 0   & b^2 & 3 d b^2\\
    0 & 0   & 0   & b^3
  \end{pmatrix}
\end{align*}
(which is also called the Bell matrix (or the transpose of the Carleman
matrix) of $g$, see
\cite{Aldrovandi:Special}) and then removing the first column. We have to solve the equation system
$A\restr_N(\beta_1^{(N)},\dots,\beta_N^{(N)})^T=(1,0,\dots,0)=:u$,
where
\[\begin{array}{lcccc}
  &A_{m,n}=\binom{n}{m} d^{n-m} b^{m}-
     I_{m,n} &&& u_m= I_{m,0}
  \\
  \\\substack{N=3}&\overbrace{\begin{pmatrix}
    d & d^2 & d^3\\
    b-1 & 2 d b & 3 d^2 b\\
    0   & b^2-1 & 3 d b^2
  \end{pmatrix}}^{n=1\dots N}
  &\begin{pmatrix}
    \beta^{(N)}_1\\\beta^{(N)}_2\\\beta^{(N)}_3
  \end{pmatrix}
  &=&\left.\begin{pmatrix}
    1\\0\\0
  \end{pmatrix}\right\}{\substack{m=0\dots N-1}}
\end{array} \]
Then we equivalently change the equation system by multiplying each row
$n$ with $s^n$
\[\begin{array}{lcccc}
  &\displaystyle A'_{m,n}=\binom{n}{m} d^{n-m} (s b)^m -
   I_{m,n}s^m&&&u'_m =  I_{m,0}
  \\\\\substack{N=3}&\begin{pmatrix}
    d & d^2 & d^3\\
    s b-s & 2d(s b) & 3d^2(s b)\\
    0   & (s b)^2-s^2 & 3d(s b)^2
  \end{pmatrix}
  &\beta^{(N)}
  &=&
  \begin{pmatrix}
    1\\0\\0
  \end{pmatrix}
\end{array}\]

As next step we equivalently change the equation system from
$A'$ to $A^{\prime\prime(N)}$ (the $N$ shall indicate that this matrix
depends on the truncation size $N$ while the previous steps where
independent of $N$)
by alternatingly adding up the lines onto the first line:
$A^{\prime\prime(N)}_0 = \sum_{n=0}^{N-1} (-1)^n A'_n$. The first line
vanishes  for $1\le m\le N-1$:
\begin{align*}
  A^{\prime\prime(N)}_{0,m} &= -(-1)^ms^m + \sum_{k=0}^{m}
  (-1)^{k}\binom{m}{k} d^{m-k} (s b)^{k}\\
  &=-(-1)^{m}s^m + (\underbrace{s(b-1)}_d-s b)^{m} = (-1)^{m+1}s^m+(-s)^m=0
\end{align*}
only the last entry at row $m=N$ is non-zero:
\begin{align*}
  A^{\prime\prime(N)}_{0,N} &= \sum_{n=0}^{m-1} (-1)^{n}\binom{m}{n}
  d^{m-n} (s b)^{k} 
  =(s(b-1)-s b)^m - (-1)^m (s b)^m 
  \\&= s^m (-1)^{m-1} ( b^m-1) = (-1)^{N-1} ((sb)^N -  I_{N,N}s^N)
\end{align*}
We multiply row 0 with $(-1)^{N-1}$ and rearrange
the lines by moving row ${n+1}$ one step up while row 0 becomes the
last row:
\[\begin{array}{lcccc}
  \substack{N=3}&
  \begin{pmatrix}
    s b-s & 2 d s b & 3 d^2 s b\\
    0   & (s b)^2-s^2 & 3d(s b)^2\\
    0  & 0  & (s b)^3-s^3
  \end{pmatrix}
  \beta^{(N)}
  &=&
  \begin{pmatrix}
    0\\0\\(-1)^2
  \end{pmatrix}
\end{array}\]
and drawing $s^m$ into $\beta$, counting now rows and columns with first index 1: 
\[\begin{array}{lcccc}
  &\displaystyle U_{m,n}=\binom{n}{m} (b-1)^{n-m} b^m- I_{m,n}&&h_m^{(N)}= I_{m,N}(-1)^{N-1}
  \\\\\substack{N=3}&\overbrace{\begin{pmatrix}
    b-1 & 2(b-1)b & 3(b-1)^2 b\\
    0   & b^2-1 & 3(b-1)b^2\\
    0  & 0  & b^3-1
  \end{pmatrix}}^{n=1\dots N}
  \begin{pmatrix}
    s\beta^{(N)}_1 \\s^2 \beta^{(N)}_2 \\s^3\beta^{(N)}_3 
  \end{pmatrix}
  &=&
  \left.\begin{pmatrix}
    0\\0\\(-1)^2
  \end{pmatrix}\right\}\substack{m=1\dots N}
\end{array}.\]

The inverse matrix of a truncation of an upper triangular matrix is equal to
the truncation of the inverse. We can give the $n$-th column of the
inverse in terms of the values of the previous columns and the $n$-th
column of the original matrix:
\begin{align*}
U^{-1}_{m,n} &= \frac{1}{U_{n,n}}\left( I_{m,n}-\sum_{k=m}^{n-1} U^{-1}_{m,k} U_{k,n}\right)
\\         &=
\frac{1}{b^n-1}\left( I_{m,n}-\sum_{k=m}^{n-1}U^{-1}_{m,k}\binom{n}{k}
  (b-1)^{n-k} b^k\right)
\end{align*}
The multiplication with $h^{(N)}$ chooses the $N$-th column
of $U^{-1}$ multiplied with sign $(-1)^{N-1}$: $\beta^{(N)}_m
=(-1)^{N-1} U^{-1}_{m,N}$, yielding the recursion
\begin{align}\label{eq:recurrence}
  s^m\beta^{(n)}_m &=
  \frac{1}{1-b^n}\left( I_{m,n}(-1)^m+\sum_{k=m}^{n-1}s^m\beta_{m}^{(k)}\binom{n}{k}
  (1-b)^{n-k} b^k\right)
\end{align}
\subsection{A direct expression of the solution}
In this subsection we apply the technique of generating functions to
obtain the direct (non-recursive) formula \eqref{eq:direct} for $\beta^{(n)}_m$. Though I
include the derivation, it is not necessary for the proof of the
formula. So the uninterested reader may skip to proposition
\ref{prop:direct} where the actual verification of the formula takes place.

We change from variable $m$ to $M$ as it will remain constant for our
further considerations and the index $m$ is needed to not run out of
variables. Multiplying \eqref{eq:recurrence} with $1-b^n$ and adding
$s^M\beta^{(n)}_Mb^n$ gives
\begin{align}\label{eq:recurrence2}
  s^M\beta^{(n)}_M &=  I_{n,M}(-1)^M+\sum_{k=M}^{n}s^M\beta_{M}^{(k)}\binom{n}{k} (1-b)^{n-k} b^k
\end{align}
We further manipulate the equations
\begin{align*}
  \underbrace{s^M \beta^{(n)}_M \frac{b^n}{n!} M!}_{T_n}
  &= I_{n,M}(-1)^Mb^n+
  b^{n}\sum_{k=M}^{n}\underbrace{s^M\beta_{M}^{(k)}\frac{b^k}{k!}M!}_{T_k}
  \frac{(1-b)^{n-k}}{(n-k)!} 
\end{align*}
to obtain the following recurrence in $T_n$
\begin{align*}
  \frac{T_n}{b^n} &=  I_{n,M}(-1)^M+\sum_{k=M}^{n}T_n  \frac{(1-b)^{n-k}}{(n-k)!}
\end{align*}
of which we consider the generating function $T(x)=\sum_{n=0}^\infty T_n
x^n$. We get the left side of the last equation as the coefficients of $T(x/b)$
\begin{align*}
  T(x/b) &= \sum_{n=0}^\infty \frac{T_n}{b^n} x^n 
\end{align*}
and the right side is the multiplication of two formal powerseries
(remember the formula $(fg)_n = \sum_{k=0}^n f_{n-k} g_k$),
namely $T$ and 
\begin{align*}
  \sum_{j=0}^\infty \frac{(1-b)^j}{j!}x^j = e^{(1-b)x}
\end{align*}
leading us to
\begin{align*}
  T(x/b) &= (-x)^M+T(x)e^{(1-b)x}.
\end{align*}
The reader may verify the following transformations for $n\ge 0$:
\begin{align*}
  T(b^{-(n+1)} x) &= (-x)^M\left(\sum_{k=0}^{n}
    \frac{e^{(1-b)x\sum_{i=k+1}^{n} b^{-i} }}{b^{kM}}\right) + T(x) e^{(1-b)x\sum_{k=0}^n b^{-k}} 
  \\ T(b^{-(n+1)} x) &= 
  (-x)^M\left(\sum_{k=0}^{n} \frac{e^{ bx\left(b^{-(n+1)} -
          b^{-(k+1)}\right)}}{b^{kM}}\right) + T(x)e^{b x (b^{-(n+1)}-1)} 
  \\ e^{-b x b^{-(n+1)}}T(b^{-(n+1)} x) -T(x)e^{-b x}  &= (-x)^M\left(\sum_{k=0}^{n}
    \frac{e^{-x b^{-k}}}{b^{kM}}\right)  
  \\ &= (-x)^M\sum_{m=0}^\infty \frac{(-x)^m}{m!}\sum_{k=0}^{n} b^{-{k m}-{k M}}
  \\ &= \sum_{m=0}^\infty
  \frac{(-x)^{m+M}}{m!}\frac{b^{-(m+M)(n+1)}-1}{b^{-(m+M)}-1}
  \\ e^{-b x y}T(y x) -T(x)e^{-b x}  &= \sum_{m=0}^\infty
  \frac{(-1)^{m+M}}{m!}\frac{y^{m+M}-x^{m+M}}{b^{-(m+M)}-1}
  \\ e^{-b x y}T(yx) -T(x)e^{-b x}  &= S(y) - 
  \underbrace{\left(\sum_{m=0}^\infty
      \frac{(-x)^{m+M}}{m!(b^{-(m+M)}-1)}\right)}_{S(x)}
\end{align*}
Letting now $y=b^{-n}\to 0$ by $n\to\infty$ for $\abs{b}>1$,
considering $S$ and $T$ being continuous at $0$ and $S(0)=T(0)=0$ we
get:
\begin{align*}
  T(x) &= e^{bx} S(x) = e^{bx} \sum_{m=M}^\infty
  \frac{(-x)^{m}}{(m-M)!(b^{-m}-1)}
\end{align*}
with the coefficients given by formal powerseries multiplication:
\begin{align*}
  T_n &= \sum_{k=M}^{n} \frac{b^{n-k}}{(n-k)!}
  \frac{(-1)^{k}}{(k-M)!(b^{-k}-1)} 
  \\&= \frac{b^{n}}{n!}M!\sum_{k=M}^{n} \binom{n}{k}\binom{k}{M}\frac{(-1)^{k}}{1-b^k} 
\end{align*}
So the direct formula is
\begin{align}\label{eq:direct}
  \beta_m^{(n)} = s^{-m}\sum_{k=m}^{n} \binom{n}{k}\binom{k}{m} \frac{(-1)^k}{1-b^k} 
\end{align}
\begin{proposition}\label{prop:direct}
  The direct expression \eqref{eq:direct} satisfies the recurrence
  \eqref{eq:recurrence} for any $b\in\C$ that is not a root of unity (i.e.\
  $b^n\neq 1$ for any $n\ge 1$), and is
  hence the solution of the truncated Abel equation system
  $A\restr_n\beta^{(n)}=u\restr_n$ of $g(x)=b\cdot(x+s)-s$
  for $s\neq 0$.
\end{proposition}
\begin{proof}
We just fill the direct expression \eqref{eq:direct} in equation
\eqref{eq:recurrence2} (which is equivalent to
the recurrence \eqref{eq:recurrence}) and show by equivalent
transformation that the recurrence is satisfied:
\begin{align*}
  s^M\beta^{(n)}_M &=
   I_{n,M}(-1)^M+\sum_{k=M}^{n}\sum_{k'=M}^{k} \binom{k}{k'}\binom{k'}{M} \frac{(-1)^{k'}}{1-b^{k'}} \binom{n}{k} (1-b)^{n-k} b^k
  \\&= I_{n,M}(-1)^M+\sum_{k'=M}^{n}\binom{k'}{M}\frac{(-1)^{k'}}{1-b^{k'}}  \sum_{k=k'}^{n}\binom{n}{k}\binom{k}{k'}  (1-b)^{n-k} b^k
  \\&= I_{n,M}(-1)^M+\sum_{k'=M}^{n}\binom{n}{k'}\binom{k'}{M}\frac{(-1)^{k'}}{1-b^{k'}}
  \underbrace{\sum_{k=k'}^{n}\binom{n-k'}{k-k'}  (1-b)^{n-k}
    b^k}_{b^{k'}\sum_{k''=0}^{n-k'} \binom{n-k'}{k''} (1-b)^{n-k'-k''}
    b^{k''}}
  \\&= I_{n,M}(-1)^M+\sum_{k'=M}^{n}\binom{n}{k'}\binom{k'}{M}\frac{(-b)^{k'}}{1-b^{k'}}
\end{align*}
Then we replace $s^M\beta^{(n)}$ on the left and subtracting the
right sum
\begin{align*}
  \sum_{k'=M}^{n}\binom{n}{k'}\binom{k'}{M}\frac{(-1)^{k'}-(-b)^{k'}}{1-b^{k'}}&= I_{n,M}(-1)^M
  \\\sum_{k'=M}^{n}\binom{n}{k'}\binom{k'}{M}(-1)^{k'}&= I_{n,M}(-1)^M
  \\\binom{n}{M}\sum_{k'=M}^{n}\binom{n-M}{k'-M}(-1)^{k'}&= I_{n,M}(-1)^M
  \\\binom{n}{M}(-1)^M0^{N-m}&= I_{n,M}(-1)^M
\end{align*}
And this is indeed a true statement, considering $0^k=0$ for $k\ge 1$
and $0^0=1$.
\end{proof}

\subsection{Convergence of the polynomial approximation}
Now, that we obtained the truncated solutions $\beta^{(N)}$, we want to
see whether the limit $\beta_n=\lim_{N\to\infty} \beta^{(N)}_n$
(according to our definition \ref{def:intuitive} of ``intuitive
solution'') exists for each $n$, which would be then the $n$-th
coefficient of the intuitive Abel function $\beta$ of $g$; where
$\alpha(x)=\beta(x-s)$ would be the $s$-intuitive Abel function of $f(x)=bx$
(compare \eqref{eq:conjugation}), which we
want to prove to be $\alpha(x)=\log_b(x)+c$ for some $c$ possibly
depending on $s$.

The coefficient wise convergence would be a direct consequence of the
pointwise convergence of the polynomial approximations 
$\alpha^{(n)}(x)=\sum_{m=1}^n \beta_m^{(n)} (x-s)^m$ to
$\alpha(x)$. In the following subsection we prove this convergence by
showing that $\tilde{\alpha}^{(n)}(x):=\alpha^{(n)}(sx)$ converges to
$\log_b(x)$. Our efforts culminate in the summarizing theorem
\ref{thm:final}.
\begin{align*}
  \alpha^{(n)}(x) &= \sum_{m=1}^\infty (x-s)^m s^{-m} \sum_{k=m}^{n} \binom{n}{k}\binom{k}{m} \frac{(-1)^k}{1-b^k} 
  \\&=\sum_{k=1}^n \binom{n}{k}\frac{(-1)^k}{1-b^k} \sum_{m=1}^k (x/s-1)^m \binom{k}{m}
  \\&=\sum_{k=1}^n \binom{n}{k}(-1)^{k+1}\frac{1-(x/s)^k}{1-b^k},
\end{align*}
which is the $s$-free function $\tilde{\alpha}^{(n)}$ applied to $x/s$:
\begin{align}
  \label{eq:log_sequence}
  \tilde{\alpha}^{(n)}(x):=\alpha(sx)=\sum_{k=1}^n \binom{n}{k}(-1)^{k+1}\frac{1-x^k}{1-b^k} .
\end{align}
With little effort,
\begin{align}
  \label{eq:onpow}
  \sum_{k=1}^n \binom{n}{k}(-1)^{k+1} y^k = 1 - (1-y)^n,
\end{align}
we can compute the value of $\tilde{\alpha}^{(n)}(b^m)$
\begin{align*}
  \tilde{\alpha}^{(n)}(b^m) &= \sum_{k=1}^n \binom{n}{k}(-1)^{k+1}\sum_{i=0}^{m-1} b^{ki}
  =\sum_{i=0}^{m-1} 1-(1-b^i)^n
  \\ \lim_{n\to\infty} \tilde{\alpha}^{(n)}(b^m) &= m
\end{align*}
confirming our hypothesis that $\lim_{n\to\infty} \tilde{\alpha}^{(n)}(x) =
\log_b(x)$. However to prove it (for a series with $f(a^n)=n$ for all
$n\ge 1$ that is not the logarithm see e.g. Euler
\cite{Euler:FalseLogarithm}) we need to show that $\lim_{n\to\infty} 
\tilde{\alpha}^{(n)}(b^x) = x$ also for non-integer $x$. 
To be careful we restrict $x$ and $b$ from here throughout this
section to $0<b<1$ and $0<x<1$. 
We have a look at the series expansion of $\frac{1-y^x}{1-y}$ for
$0<y<1$:
\begin{align}
  1-y^x &= 1-\sum_{j=0}^\infty \binom{x}{j} (y-1)^j
              = -\sum_{j=1}^\infty \binom{x}{j} (y-1)^j
  \\\frac{1-y^x}{1-y}&=\sum_{j=0}^\infty \binom{x}{j+1} (y-1)^j\label{eq:frac}
  \\ &= \sum_{j=0}^\infty \binom{x}{j+1} \sum_{i=0}^j \binom{j}{i} (-1)^{j-i}y^i
\end{align}
Now substituting $y=b^k$
\begin{align}
  \tilde{\alpha}^{(n)}(b^x) &= \sum_{k=1}^n \binom{n}{k}(-1)^{k+1}\frac{1-(b^k)^x}{1-b^k} \label{eq:sum}
\end{align}
And knowing that $\sum_{k=1}^n \binom{n}{k}(-1)^{k+1} b^{ki} = 1 -
(1-b^i)^n$ we write
\begin{align*}
  \tilde{\alpha}^{(n)}(b^x)  &= \sum_{j=0}^\infty \binom{x}{j+1} \sum_{i=0}^j \binom{j}{i} (-1)^{j-i} (1-(1-b^{i})^n) .
\end{align*}
We split the sums into two parts at the first minus of $1 -
(1-b^i)^n$. Considering $\sum_{i=0}^j \binom{j}{i} (-1)^{j-i} 1 =
0^j$, where $0^j=0$ for $j\ge 1$ and $0^0=1$, we get
\begin{align*}
  \tilde{\alpha}^{(n)}(b^x) &=  x-\sum_{j=1}^\infty \binom{x}{j+1} \underbrace{\sum_{i=0}^j \binom{j}{i} (-1)^{j-i} (1-b^{i})^n }_{R_j^{(n)}}.
\end{align*}
Obviously $\lim_{n\to\infty} R^{(n)}_j = 0$ for $j\ge 1$ because 
$\abs{1-b^i}<1$ for each $i$. In the
remainder of this section we show that the sequence (of sequences) $R^{(n)}$ converges
not only point-wise but uniformly to $(0,0,\dots)$ in the supremum norm $||v|| =
\sup_{j\in\N}\abs{v_j}$ which then implies that we can
swap taking the limit in $n$ with the limit in $j$.

So we show that for each $\eps>0$ there is an $n_0$ such that
$\abs{R^{(n)}_j}<\eps$ for all $j$ and $n>n_0$. For $j=1$ we find an
$n=n_0$ such that
\begin{align*}
  \abs{R^{(n)}_j} \le \underbrace{\sum_{i=0}^j \binom{j}{i}
    \abs{1-b^{i}}^{n}}_{=:d_{j,n}} <\eps
\end{align*}
Noticing that $d_{j,n}$ is decreasing in the second index $n$ the
above equation is also valid for any $n\ge n_0$.

By binomially expanding the power with exponent $n$ and unexpand it
into a power with exponent $j$ we reformulate the expression of
$R^{(n)}_j$ to:
\begin{align*}
  R^{(n)}_j = (-1)^{j-n}\sum_{k=0}^n \binom{n}{k}(-1)^{n-k} (1-b^k)^j
\end{align*}
and so obtain that $\abs{R^{(n)}_j} = \abs{R^{(j)}_n} \le d_{n,j}$ where
$d_{n,j}$ is decreasing in the second argument which is now $j$. Hence $\abs{R^{(n)}_j}<\eps$ for all
$n\ge n_0$ and $j\ge 1$. 

Now, to finish the proof we show that 
\begin{align*}
  \lim_{n\to\infty}\sum_{j=1}^\infty \binom{x}{j+1} R^{(n)}_j = 0
\end{align*}
As prerequisite we need the well known
\begin{preliminary}
  $\sum_{n=0}^\infty \binom{\kappa}{n} z^n$ converges absolutely (to
  $(z+1)^\kappa$) for all $z$ with $\abs{z}=1$ if $\Re(\kappa) > 0$.
\end{preliminary}
\begin{proof}
Done via comparison with the Riemann zeta function by the
inequality $\abs{\binom{\kappa}{k}} \le
\frac{e^{\abs{\kappa}^2+\Re(\kappa)}}{k^{1+\Re(\kappa)}}$, $k\ge 1$.
\end{proof}
\begin{corollary}
  The limit $t_\kappa:=\sum_{j=1}^\infty \abs{\binom{\kappa}{j+1}}$ exists for every $\kappa>0$.
\end{corollary}
\begin{proof}
  Follows from letting $z=1$ in the preliminary. 
\end{proof}

Then for a given $\eps>0$ choose $n_0$ such that
$R^{(n)}_j <\eps/t_x$ for all $j\ge 1$ and $n\ge n_0$ and obtain:
\begin{align*}
  \abs{\sum_{j=1}^\infty \binom{x}{j+1} R^{(n)}_j }
  < {\sum_{j=1}^\infty \abs{\binom{x}{j+1}} \frac{\eps}{t_x} }
  =\eps\mt{for}n\ge n_0
\end{align*}
So we have established that $\lim_{n\to\infty} \tilde{\alpha}^{(n)}(b^x) = x$ for
$x>0$. Which has the consequence that $\tilde{\alpha}(y):=\lim_{n\to\infty} \tilde{\alpha}^{(n)}(y)$
exists for all $0<y<1$ and is the inverse function of $x\mapsto b^x$.

We finally summarize all our findings:
\begin{theorem}\label{thm:final}
   The (unique) polynomial $\beta^{(n)}(z)$ of degree $n$ that satisfies
   $\beta^{(n)}(0)=0$ and the Abel equation
   \begin{align*}
     \beta^{(n)}(b\cdot(z+s)-s) &= 1+\beta^{(n)}(z) \quad (s\neq
     0,\; b^k\neq 1\forall k\ge 1)
   \end{align*}
   is given by $\beta^{(n)}(z)=\alpha^{(n)}(z+s)=\tilde{\alpha}^{(n)}(z/s+1)$ where
   \begin{align*}
      \tilde{\alpha}^{(n)}(z)&=\sum_{k=1}^n \binom{n}{k}(-1)^{k+1}\frac{1-z^k}{1-b^k} 
   \end{align*}
   and for all $x,b\in(0,1)$ we have the convergence:
   \begin{align}
    \lim_{n\to\infty} \tilde{\alpha}^{(n)}(x)=\log_b(x).
  \end{align}
  $\alpha(x)=\tilde{\alpha}(x/s)=\log_b(x)-\log_b(s)$ is the $s$-intuitive Abel
  function of $f(x)=bx$. (It is independent on
  $s$ up to an additive constant.)
\end{theorem}

\section{Comments}
The most urgent questions to develop the mathematics of the intuitive method are already
listed in the introduction. Here only some side notes:

Numerically it appears that convergence of $\tilde{\alpha}$ is also achieved for
$\abs{x/b-1}<1$ in the case $b>1$ which points towards possible
improvements of the theorem.

The more interesting question about the convergence of the approximating
polynomials of the intuitive Abel function of $f(x)=e^x$ seems out of
reach to solve with these rather elementary techniques. Numerically at
least it seems that the coefficients do  not converge
uniformly but have a point-wise limit which is invariant under shift
conjugations.

The author thanks for the stimulating discussions on
\cite{Tetrationforum}. Without them this paper would not have come into
existence.

\bibliography{main}

\end{document}